\documentclass[12pt,twoside]{article}
\usepackage[latin1]{inputenc}
\usepackage[english]{babel}
\usepackage{clm-macros}
\usepackage{amsmath}
\usepackage{psfrag}
\usepackage{epsfig}
\usepackage{graphicx}
\usepackage{verbatim}
\usepackage{tikz}
\usepackage{mathtools}
\usepackage{datetime}
\usepackage{fullpage}
\usepackage{enumerate}
\usepackage{subfigure}

\usetikzlibrary{decorations}
\usetikzlibrary{decorations.pathmorphing}
\usetikzlibrary{arrows}
\tikzstyle{every node}=[circle, draw, fill=white,inner sep=0pt, minimum width=6pt]
\tikzstyle{nodelabel}=[rounded corners,fill=none,inner sep=5pt,draw=none]
\tikzstyle{matching} = [ultra thick]
\tikzset{snake/.style={decorate, decoration=snake}}

\newdateformat{UKvardate}{
\THEDAY\ \monthname[\THEMONTH], \THEYEAR}

\def\Notname{Notation}
\def\casname{Case}
\def\cnjname{Conjecture}
\def\corname{Corollary}
\def\Defname{Definition}
\def\exename{Exercise}
\def\exaname{Example}
\def\lemname{Lemma}
\def\listassertionname{List of Assertions}
\def\prbname{Problem}
\def\prpname{Proposition}
\def\quename{Question}
\def\proofOfname{\proofname{} of}
\def\ssbname{Case}
\def\subname{Case}
\def\thmname{Theorem}

\begin{document}

\newcommand{\pema}{perfect matching}
\newcommand{\mcg}{matching covered graph}
\newcommand{\mc}{matching covered}
\newcommand{\ec}{equivalence class}
\newcommand{\rc}{removable class}
\newcommand{\ecp}{equivalence class partition}
\newcommand{\tcd}{tight cut decomposition}
\newcommand{\scd}{separating cut decomposition}

\title{Equivalence classes in matching covered graphs\footnote{This research is supported by
National Natural Science Foundation of China (Grant No. 11671186); by NSF of Shandong Province, China (ZR2013LA007); and by FAPESP of S{\~a}o Paulo, Brazil (2018/04679-1).}}
\author{Fuliang Lu$^{a}$
            \and
            Nishad Kothari$^{b}$
            \and
            Xing Feng$^{c}$
            \and
            Lianzhu Zhang$^{d}$\\
{\small $^{a}$ School of Mathematics and Statistics, Minnan Normal University, Zhangzhou,  China}\\
{\small $^{b}$ Institute of Computing, University of Campinas, Campinas, Sao Paulo, Brazil}\\
{\small $^{c}$ Faculty of Science, Jiangxi University of Science and Technology, Ganzhou, China}\\
{\small $^{d}$ School of Mathematical Science, Xiamen University, Xiamen, Fujian, China}}

\UKvardate
\date{14 December, 2019}
  \maketitle
  \thispagestyle{empty}

\begin{abstract}
A connected graph~$G$, of order two or more, is \mc\ if each edge lies in some \pema.
The tight cut decomposition of a \mcg~$G$ yields a list of bricks and braces; as per a theorem
of Lov{\'a}sz~\cite{lova87}, this list is unique (up to multiple edges); $b(G)$ denotes
the number of bricks, and $c_4(G)$ denotes the number of braces that are isomorphic to
the cycle~$C_4$
(up to multiple edges).

\smallskip
Two edges $e$ and $f$ are mutually dependent if, for each \pema~$M$,
$e \in M$ if and only if $f \in M$;
Carvalho, Lucchesi and Murty investigated this notion in their landmark paper~\cite{clm99}.
For any \mcg~$G$, mutual dependence is an equivalence
relation, and it partitions $E(G)$ into {\ec}es;
this \ecp\ is denoted by $\mathcal{E}_G$ and we refer to its parts as {\ec}es
of~$G$; we use $\varepsilon(G)$ to denote the cardinality of the largest \ec.

\smallskip
The operation of `splicing' may be used to construct bigger {\mcg}s from smaller ones;
see~\cite{lckm18}; `tight splicing' is a stronger version of `splicing'.
(These are converses of the notions of `separating cut' and `tight cut'.)
In this article, we answer the following basic question:
if a \mcg~$G$ is obtained by `splicing' (or by `tight splicing')
two smaller {\mcg}s, say~$G_1$~and~$G_2$,
then how is $\mathcal{E}_G$ related to~$\mathcal{E}_{G_1}$ and to~$\mathcal{E}_{G_2}$
(and vice versa)?

\smallskip
As applications of our findings: firstly, we establish tight upper bounds on $\varepsilon(G)$
in terms of $b(G)$ and $c_4(G)$; secondly, we answer a recent question of He, Wei,
Ye and Zhai~\cite{hwyz19}, in the affirmative, by constructing graphs that have arbitrarily high
$\kappa(G)$~and~$\varepsilon(G)$ simultaneously,
where $\kappa(G)$ denotes the vertex-connectivity.
\end{abstract}

\tableofcontents

\section{Matching covered graphs}

For general graph-theoretic terminology, we follow Bondy and Murty~\cite{bomu08}.
All graphs considered here are finite and loopless; however, we do allow multiple edges.
We begin by reviewing some important terminology and notation.

\smallskip
For a subset $X$ of the vertex set~$V(G)$ of a graph~$G$, we use $N_G(X)$, or simply~$N(X)$,
to denote
the set of vertices that have at least one neighbor in~$X$; and we use ${\rm odd}(G-X)$
to denote the number of components of odd order in the graph~$G-X$.
Furthermore, we use
$\partial_G(X)$, or simply~$\partial(X)$,
to denote the set of edges of~$G$
that have exactly one end in~$X$; such a set $\partial(X)$ is called
a {\it cut} of~$G$;
the sets $X$ and $\overline{X}:=V(G)-X$ are referred to as the {\it shores} of~$\partial(X)$.
For a vertex~$v$, we simplify the notation~$\partial(\{v\})$ to $\partial(v)$.
A cut~$C$ is {\it trivial} if either shore has just one vertex; otherwise $C$ is {\it nontrivial}.
A cut~$C$ is a {\it $k$-cut} if $|C|=k$.

\smallskip
A graph is {\it matchable} if it has a \pema.
Tutte's $1$-factor Theorem states that a graph~$G$ is matchable if and only ${\rm odd}(G-S) \leq |S|$
for every $S \subseteq V(G)$.
For a graph~$G$, a subset $B \subseteq V(G)$
is called a {\it barrier} if ${\rm odd}(G-B) = |B|$.

\smallskip
An edge~$e$ of a graph~$G$ is {\it admissible} if it lies in some \pema\ of~$G$;
otherwise $e$ is {\it inadmissible}.
The following is easily proved using Hall's Theorem; see Figure~\ref{fig:bip-inadmissible}.
\begin{lem}
\label{lem:bip-inadmissible}
Let $H[A,B]$ denote a bipartite matchable graph. An edge~$e$ of~$H$ is inadmissible
if and only if there exists a nonempty proper subset~$S \subset A$ such that $|N_H(S)| = |S|$
and $e$ has one end in~$N_H(S)$ and its other end is not in~$S$. \qed
\end{lem}

A connected graph, of order two or more, is {\it \mc} if each edge is admissible.
The following fact is easily deduced from Lemma~\ref{lem:bip-inadmissible}.

\begin{prp}
\label{prp:bip-mcg-characterization}
For a bipartite matchable graph~$H[A,B]$, of order four or more, the following statements
are equivalent:
\begin{enumerate}[(i)]
\item $H$ is \mc.
\item $|N_H(S)| \geq |S|+1$ for every nonempty proper subset~$S$ of~$A$.
\item $H-a-b$ is matchable for each pair of vertices $a \in A$ and $b \in B$. \qed
\end{enumerate}
\end{prp}

Using Tutte's $1$-factor Theorem, one may prove that a connected matchable graph~$G$
is \mc\ if and only if every barrier of~$G$ is stable (that is, an independent set).
Kotzig proved the following fundamental theorem: the maximal barriers of a
\mcg~$G$ partition its vertex-set~$V(G)$; this partition of~$V(G)$
is called the {\it canonical partition} of~$G$.
(See Lov{\'a}sz and Plummer~\cite[page 150]{lopl86}.)

\smallskip
In this article,
our main focus is the `equivalence class partition' --- a partition
of the edge-set~$E(G)$ of a \mcg~$G$ --- that was formally introduced and investigated by
Carvalho, Lucchesi and Murty in their landmark paper~\cite{clm99}.

\begin{figure}[!htb]
\centering
\subfigure[$e$ is inadmissible]
{
\begin{tikzpicture}[scale=0.7]
\draw (0,0) -- (4,0) -- (4,1) -- (0,1) -- (0,0);
\draw (6,0) -- (10,0) -- (10,1) -- (6,1) -- (6,0);
\draw (2,0.5)node[below,nodelabel]{$N(S)$};

\draw (0,3) -- (4,3) -- (4,4) -- (0,4) -- (0,3);
\draw (6,3) -- (10,3) -- (10,4) -- (6,4) -- (6,3);
\draw (2,4)node[above,nodelabel]{$S$};

\draw (3.7,1.8)node[nodelabel]{$e$};

\draw[ultra thin] (3.5,0.5) -- (7.5,3.5);
\draw[ultra thin] (2.5,0.5) -- (6.5,3.5);

\draw[ultra thin] (1.5,0.5) -- (1.5,3.5);
\draw[ultra thin] (0.5,0.5) -- (0.5,3.5);
\draw[ultra thin] (8.5,0.5) -- (8.5,3.5);
\draw[ultra thin] (9.5,0.5) -- (9.5,3.5);
\end{tikzpicture}
\label{fig:bip-inadmissible}
}
\hspace*{0.5in}
\subfigure[$\partial(X)$ is a tight cut; $B \cap X$ is a barrier]
{
\begin{tikzpicture}[scale=0.7]

\draw (0,0) -- (4,0) -- (4,-1) -- (0,-1) -- (0,0);
\draw (2,-0.75) node[above,nodelabel]{$B \cap X$};

\draw (0,-2.5) -- (3.2,-2.5) -- (3.2,-3.5) -- (0,-3.5) -- (0,-2.5);
\draw (1.6,-2.75) node[below,nodelabel]{$A \cap X$};

\draw (5.2,-2.5) -- (9.2,-2.5) -- (9.2,-3.5) -- (5.2,-3.5) -- (5.2,-2.5);
\draw (7.2,-2.75) node[below,nodelabel]{$A \cap \overline{X}$};

\draw (6,0) -- (9.2,0) -- (9.2,-1) -- (6,-1) -- (6,0);
\draw (7.6,-0.75) node[above,nodelabel]{$B \cap \overline{X}$};

\draw[ultra thin] (1,-0.5) -- (1,-3);
\draw[ultra thin] (1.7,-0.5) -- (1.7,-3);
\draw[ultra thin] (2.4,-0.5) -- (6.1,-3);
\draw[ultra thin] (3.1,-0.5) -- (6.8,-3);
\draw[ultra thin] (7.5,-3) -- (7.5,-0.5);
\draw[ultra thin] (8.2,-3) -- (8.2,-0.5);

\end{tikzpicture}
\label{fig:bip-tight-cut}
}

\vspace*{0.25in}
\subfigure[$e$ depends on $f$]
{
\begin{tikzpicture}[scale=0.7]
\draw (0,0) -- (4,0) -- (4,1) -- (0,1) -- (0,0);
\draw (6,0) -- (10,0) -- (10,1) -- (6,1) -- (6,0);
\draw (2,0)node[below,nodelabel]{$A_0$};
\draw (8,0)node[below,nodelabel]{$A_1$};

\draw (0,3) -- (4,3) -- (4,4) -- (0,4) -- (0,3);
\draw (6,3) -- (10,3) -- (10,4) -- (6,4) -- (6,3);
\draw (2,4)node[above,nodelabel]{$B_0$};
\draw (8,4)node[above,nodelabel]{$B_1$};

\draw[ultra thick] (3.5,3.5) -- (6.5,0.5);
\draw (4.7,3)node[nodelabel]{$f$};
\draw (3.7,1.8)node[nodelabel]{$e$};

\draw[ultra thin] (3.5,0.5) -- (7.5,3.5);
\draw[ultra thin] (2.5,0.5) -- (6.5,3.5);

\draw[ultra thin] (1.5,0.5) -- (1.5,3.5);
\draw[ultra thin] (0.5,0.5) -- (0.5,3.5);
\draw[ultra thin] (8.5,0.5) -- (8.5,3.5);
\draw[ultra thin] (9.5,0.5) -- (9.5,3.5);
\end{tikzpicture}
\label{fig:bip-dependence}
}
\hspace*{0.5in}
\subfigure[$e$ and $f$ are mutually dependent]
{
\begin{tikzpicture}[scale=0.7]
\draw (0,0) -- (4,0) -- (4,1) -- (0,1) -- (0,0);
\draw (6,0) -- (10,0) -- (10,1) -- (6,1) -- (6,0);
\draw (2,0)node[below,nodelabel]{$A_0$};
\draw (8,0)node[below,nodelabel]{$A_1$};

\draw (0,3) -- (4,3) -- (4,4) -- (0,4) -- (0,3);
\draw (6,3) -- (10,3) -- (10,4) -- (6,4) -- (6,3);
\draw (2,4)node[above,nodelabel]{$B_0$};
\draw (8,4)node[above,nodelabel]{$B_1$};

\draw[ultra thick] (3.5,3.5) -- (6.5,0.5);
\draw (4.7,3)node[nodelabel]{$f$};


\draw[ultra thick] (3.5,0.5) -- (6.5,3.5);
\draw (4.6,1.1)node[nodelabel]{$e$};

\draw[ultra thin] (1.5,0.5) -- (1.5,3.5);
\draw[ultra thin] (0.5,0.5) -- (0.5,3.5);
\draw[ultra thin] (8.5,0.5) -- (8.5,3.5);
\draw[ultra thin] (9.5,0.5) -- (9.5,3.5);
\end{tikzpicture}
\label{fig:bip-mutual-dependence}
}
\caption{Illustrations of various concepts in bipartite graphs}
\label{fig:bip}
\end{figure}
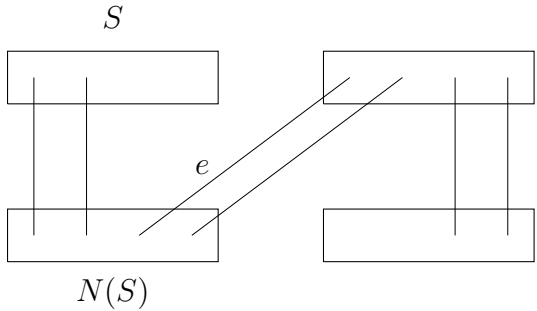
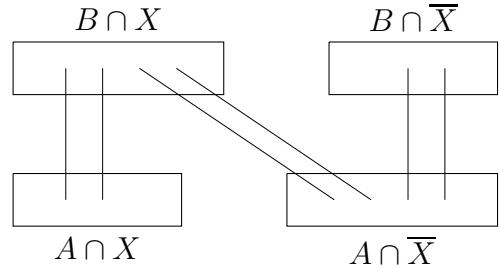
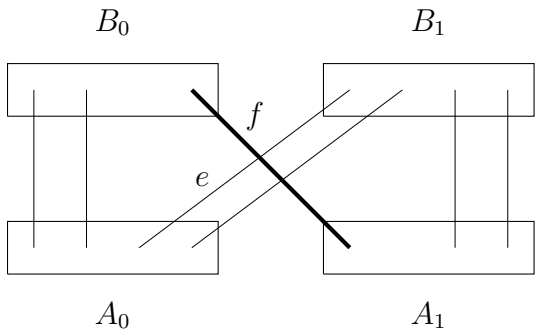
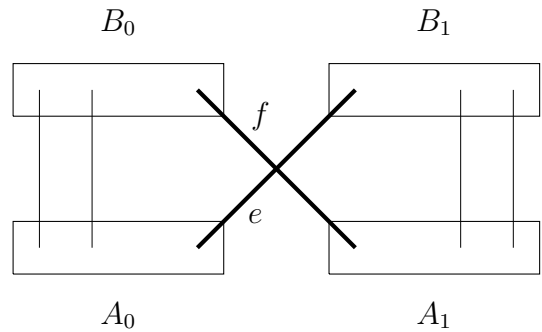

\subsection{The equivalence class partition $\mathcal{E}_G$}

For a cut $C:=\partial(X)$ of a \mcg~$G$, the parities of~$|X|$ and $|\overline{X}|$
are the same; we say that $C$ is an {\it odd cut} if $|X|$ is odd; otherwise, $C$ is an {\it even cut}.

\smallskip
For a \mcg~$G$, an edge~$e$ {\it depends on} an edge~$f$, denoted as
$e \xrightarrow{G} f$, if each \pema\
that contains $e$ also contains~$f$.
Note that, for two distinct edges $e$ and $f$, $e \xrightarrow{G} f$
if and only if $e$ is inadmissible in~$G-f$.
The following is easily deduced from Lemma~\ref{lem:bip-inadmissible}
and Proposition~\ref{prp:bip-mcg-characterization}; see Figure~\ref{fig:bip-dependence}.

\begin{cor}
\label{cor:bip-dependence}
Let $e$ and $f$ denote distinct edges of a bipartite \mcg\ $H[A,B]$.
Then $e$ depends on~$f$ if and only if there exist partitions $(A_0,A_1)$ of~$A$
and $(B_0,B_1)$ of~$B$ such that (i) $|A_0| = |B_0|$, (ii) $e$ joins a vertex in~$A_0$
to a vertex in~$B_1$, and (iii) $f$ is the only edge that joins a vertex in~$B_0$
to a vertex in~$A_1$. \qed
\end{cor}

\smallskip
For a \mcg~$G$, edges $e$ and $f$ are {\it mutually dependent}, denoted as
$e \xleftrightarrow{G} f$, if $e \xrightarrow{G} f$ and $f \xrightarrow{G} e$.
For example:
one may verify that if $\{e,f\}$ is an even $2$-cut of a bipartite \mcg~$H$, then $H-e-f$
has precisely two (matchable) components and $e \xleftrightarrow{H} f$.
The converse also holds, and is easily deduced from Corollary~\ref{cor:bip-dependence};
see Figure~\ref{fig:bip-mutual-dependence}.

\begin{cor}
\label{cor:bip-mutual-dependence}
Let $e$ and $f$ denote distinct edges of a bipartite \mcg~$H$.
The following are equivalent:
\begin{enumerate}[(i)]
\item $e \xleftrightarrow{G} f$.
\item $\{e,f\}$ is an even $2$-cut of~$H$. \qed
\end{enumerate}
\end{cor}

Observe that
{\it mutual dependence} is an equivalence relation;
whence it partitions the edge set $E(G)$ into {\ec}es.
Throughout this article, we refer to this partition of~$E(G)$,
denoted as $\mathcal{E}_G$,
as the {\it \ecp} of a \mcg~$G$,
and we refer to its parts as the {\it {\ec}es} of~$G$;
we also use the notation~$\varepsilon(G)$ to denote the cardinality
of the largest member of~$\mathcal{E}_G$.
Corollary~\ref{cor:bip-mutual-dependence} yields the following consequence.

\begin{cor}
\label{cor:bip-ec}
For a bipartite \mcg~$H$, the following are equivalent:
\begin{enumerate}[(i)]
\item $\varepsilon(H) = 1$.
\item $H$ is free of even $2$-cuts. \qed
\end{enumerate}
\end{cor}

An \ec~$F$ is a {\it singleton} if $|F|=1$, and it is a {\it doubleton} if $|F|=2$.
Thus, in the case of bipartite graphs, edge-connectivity three (or more) already implies
that each \ec\ is a singleton. This is in stark contrast with the case of
nonbipartite graphs ---
as we will see in Section~\ref{sec:arbitrarily-high-kappa-and-varepsilon}.
The reader may verify that, for each of the graphs shown in Figure~\ref{fig:nonbip},
the sets $\{e_1,e_2\}, \{f_1,f_2\}$ and $\{g_1,g_2\}$ are
doubleton {\ec}es;
all other {\ec}es
are singleton.

\smallskip
An edge~$e$ of a \mcg~$G$ is {\it removable} if the graph~$G-e$ is also \mc.
Using the fact that all {\mcg}s, distinct from~$K_2$, are $2$-edge-connected,
one may easily verify the following.
\begin{lem}
For every \mcg~$G$, distinct from~$K_2$, an edge~$e$ is {\it removable}
if and only if no other edge depends on~$e$. \qed
\end{lem}

Consequently, every removable edge is the member of a
singleton {\ec}. It is worth noting that every multiple edge is in fact a removable edge.
Removable edges play an important role in the theory of {\mcg}s, and especially
in several works of Carvalho, Lucchesi and Murty~\cite{clm99}.
We will further discuss removable edges in Section~\ref{sec:tight-cuts-bricks-braces}.

\smallskip
The following provides another (equivalent) definition of {\ec}es of a \mcg,
and is easily verified.

\begin{prp}
\label{prp:ec-alternative-definition}
For a \mcg~$G$ and a subset $F \subseteq E(G)$, the following are equivalent:
\begin{enumerate}[(i)]
\item $F$ is a (not necessarily proper) subset of some member of~$\mathcal{E}_G$.
\item For each \pema~$M$ of~$G$, either $F \subseteq M$ or $F \cap M = \emptyset$. \qed
\end{enumerate}
\end{prp}

We infer that $\varepsilon(G) \leq n$
for any \mcg~$G$ of order~$2n$.
Observe that $\mathcal{E}_{C_{2n}}$
has precisely two members, each of which is a perfect matching of the even cycle~$C_{2n}$.
Consequently, $\varepsilon(C_{2n}) = n$.
We leave the following as an easy
exercise\footnote{Hint: fix a \pema~$M$ that is an \ec\ (of~$G$) using
Proposition~\ref{prp:ec-alternative-definition}, and consider the symmetric difference
of $M$ and any other {\pema}.}
for the reader to get acquainted
with the notion of {\ec}es.

\begin{prp}
\label{prp:ec-corresponding-to-pm}
Let $G$ denote a \mcg\ of order~$2n$, where $n \geq 2$.
Then \mbox{$\varepsilon(G) = n$} if and only if
(i) the underlying simple graph of $G$ is either $K_4$ or $C_{2n}$ and
(ii)~$G$~has a perfect matching~$M$, each of whose member has multiplicity precisely one. \qed
\end{prp}

An \ec~$R$ of a \mcg~$G$ is a {\it \rc} if the graph \mbox{$G-R$} is also \mc.
Since a \mcg\ is connected (by definition),
Corollary~\ref{cor:bip-mutual-dependence} implies
that each removable class of a bipartite \mcg\ is a singleton.
The following is a consequence of
a Lov{\'a}sz and Plummer~\cite[Lemma 5.4.5]{lopl86}.

\begin{thm}
\label{thm:rc-singleton-or-doubleton}
In every \mcg, each \rc\ is either a singleton or a doubleton.~\qed
\end{thm}

Now, let~$R$ denote a removable class of a \mcg~$G$. Observe that $R$ is a singleton
if and only if its (only) member is a removable edge.
On the other hand,
if~$R$ is a doubleton, then $R$ is referred to as a {\it removable doubleton}.
Removable edges and removable doubletons play a crucial role in the theory
of {\mcg}s; see~\cite{clm99}. (It is worth noting that a \mcg\ may have singleton or doubleton
{\ec}es that are not removable.)

\subsection{Splicing and separating cuts}

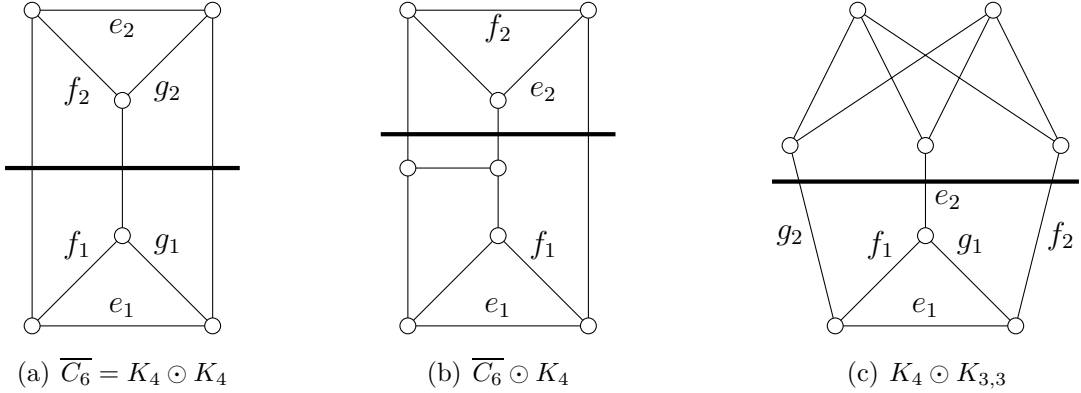
\begin{figure}[!htb]
\centering
\subfigure[$\overline{C_6} = K_4 \odot K_4$]
{
\begin{tikzpicture}[scale=1.2]

\draw (0,0.5) -- (0,-1);
\draw (-1,1.5) -- (-1,-2);
\draw (1,1.5) -- (1,-2);

\draw (0,0.5) -- (-1,1.5) -- (1,1.5) -- (0,0.5);

\draw (0,-1) -- (-1,-2) -- (1,-2) -- (0,-1);

\draw (0,0.5)node{};
\draw (-1,1.5)node{};
\draw (1,1.5)node{};

\draw (0,-1)node{};
\draw (-1,-2)node{};
\draw (1,-2)node{};

\draw (0,-1.8)node[nodelabel]{$e_1$};
\draw (0,1.3)node[nodelabel]{$e_2$};
\draw (-0.5,-1.1)node[nodelabel]{$f_1$};
\draw (-0.5,0.6)node[nodelabel]{$f_2$};
\draw (0.5,-1.1)node[nodelabel]{$g_1$};
\draw (0.5,0.6)node[nodelabel]{$g_2$};

\draw[ultra thick] (-1.3,-0.25) -- (1.3,-0.25);

\end{tikzpicture}
\label{fig:C6bar}
}
\hspace*{0.5in}
\subfigure[$\overline{C_6} \odot K_4$]
{
\begin{tikzpicture}[scale=1.2]

\draw (0,-0.25) -- (-1,-0.25);

\draw (0,0.5) -- (0,-1);
\draw (-1,1.5) -- (-1,-2);
\draw (1,1.5) -- (1,-2);

\draw (0,0.5) -- (-1,1.5) -- (1,1.5) -- (0,0.5);

\draw (0,-1) -- (-1,-2) -- (1,-2) -- (0,-1);

\draw (0,0.5)node{};
\draw (-1,1.5)node{};
\draw (1,1.5)node{};

\draw (0,-1)node{};
\draw (-1,-2)node{};
\draw (1,-2)node{};

\draw (0,-0.25)node{};
\draw (-1,-0.25)node{};

\draw (0,-1.8)node[nodelabel]{$e_1$};
\draw (0,1.3)node[nodelabel]{$f_2$};
\draw (0.5,-1.1)node[nodelabel]{$f_1$};
\draw (0.5,0.6)node[nodelabel]{$e_2$};

\draw[ultra thick] (-1.3,0.125) -- (1.3,0.125);

\end{tikzpicture}
\label{fig:Bicorn}
}
\hspace*{0.5in}
\subfigure[$K_4 \odot K_{3,3}$]
{
\begin{tikzpicture}[scale=1.2]

\draw (0.75,1.5) -- (-1.5,0) -- (-0.75,1.5) -- (0,0) -- (0.75,1.5) -- (1.5,0) -- (-0.75,1.5);

\draw (0,0) -- (0,-1);
\draw (-1.5,0) -- (-1,-2);
\draw (1.5,0) -- (1,-2);

\draw (0,-1) -- (-1,-2) -- (1,-2) -- (0,-1);

\draw (0.75,1.5)node{};
\draw (-0.75,1.5)node{};

\draw (0,0)node{};
\draw (-1.5,0)node{};
\draw (1.5,0)node{};

\draw (0,-1)node{};
\draw (-1,-2)node{};
\draw (1,-2)node{};

\draw (0,-1.8)node[nodelabel]{$e_1$};
\draw (0.25,-0.6)node[nodelabel]{$e_2$};
\draw (-0.5,-1.1)node[nodelabel]{$f_1$};
\draw (1.5,-1)node[nodelabel]{$f_2$};
\draw (0.5,-1.1)node[nodelabel]{$g_1$};
\draw (-1.5,-1)node[nodelabel]{$g_2$};

\draw[ultra thick] (-1.7,-0.4) -- (1.7,-0.4);

\end{tikzpicture}
\label{fig:K4-splice-K33}
}
\caption{A few nonbipartite {\mcg}s}
\label{fig:nonbip}
\end{figure}

For $i \in \{1,2\}$, let $G_i$ denote a graph with a specified vertex~$v_i$
so that (i) $G_1$~and~$G_2$ are disjoint and (ii) degree of $v_1$ in~$G_1$
is the same as degree of~$v_2$ in~$G_2$. Suppose that $\pi$~is a bijection
from~$\partial_{G_1}(v_1)$ to $\partial_{G_2}(v_2)$.
We denote by $(G_1 \odot G_2)_{v_1,v_2,\pi}$ the graph obtained from the
union of~$G_1-v_1$ and $G_2 - v_2$ by joining, for each edge $e \in \partial_{G_1}(v_1)$,
the end of~$e$ in~$G_1-v_1$ to the end of $\pi(e)$ in $G_2 - v_2$.
We refer to $(G_1 \odot G_2)_{v_1,v_2,\pi}$ as the graph obtained by
{\it splicing $G_1$~at~$v_1$ with $G_2$~at~$v_2$ with respect to the bijection~$\pi$},
or simply as a graph obtained by {\it splicing $G_1$ and $G_2$}.
Thus, $V(G) = (V(G_1) - v_1) \cup (V(G_2) - v_2)$; the corresponding
cut~$C:=\partial(V(G_1)-v_1)$ is referred to as the {\it splicing cut};
the vertices $v_1$~and~$v_2$ are referred to as the {\it splicing vertices}.
The following is easy to see.

\begin{lem}
\label{lem:splicing-simple}
If a graph~$G$ is obtained by splicing two simple graphs, then $G$ is simple, and the corresponding
splicing cut is a matching of~$G$. \qed
\end{lem}

The following basic, nonetheless crucial, fact is easily proved.

\begin{lem}
\label{lem:splicing-mcg}
Any graph obtained by splicing two {\mcg}s is also \mc. \qed
\end{lem}

Figure~\ref{fig:nonbip} shows a few examples that are obtained by splicing smaller {\mcg}s;
the splicing cut is depicted using a thick line.
(In all of these examples, the choice of the splicing vertices and the choice of the permutation
does not matter; it is for this reason that we have simplified the notation.) However, in general,
these choices do matter. For instance, splicing two copies of the wheel~$W_5$ (at their hubs)
can result in several non-isomorphic graphs
(such as the Petersen graph and
the pentagonal prism)
depending on the choice of the permutation.

\smallskip
Given any cut~$C:=\partial_G(X)$ of a graph~$G$,
we denote by~$G/X \rightarrow x$, or simply by~$G/X$, the graph obtained from~$G$
by shrinking $X$ to a single vertex~$x$ (and deleting any resulting loops);
we refer to $x$ as the {\it contraction vertex}.
The two graphs $G/X$ and $G/\overline{X}$
are called the \mbox{\it $C$-contractions of~$G$}.
(Observe that $G$ may be recovered by splicing
the $C$-contractions, at their contraction vertices, appropriately.)
The following lemma is easily proved, and it will be useful to us later.

\begin{lem}
\label{lem:preservation-of-even-cuts}
Let $C:=\partial(X)$ denote an odd cut of a connected graph~$G$ of even order,
and let $H$ denote a $C$-contraction of~$G$. Then: (i) $H$ is of even order,
(ii) if $D$ is a cut of~$H$ then $D$ is a cut of~$G$, and (iii) for any cut~$D$ of the graph~$H$,
$D$ is an even cut of~$H$ if and only if $D$ is an even cut of~$G$.~\qed
\end{lem}

\smallskip
A cut~$C$ of a \mcg~$G$ is a {\it separating cut}
if both $C$-contractions of~$G$ are also \mc.
Consequently, $G$ has a nontrivial separating cut $C$ if and only if
$G$ can be obtained by splicing two smaller {\mcg}s (and the corresponding splicing cut
is precisely the separating cut~$C$).
The following is easy to see.

\begin{lem}
\label{lem:sep-cut-contraction-pm-extension}
Let $C$ denote a separating cut of a \mcg~$G$,
and let $G_1$ denote a $C$-contraction of~$G$.
Then each \pema\ of~$G_1$ extends to a \pema\ of~$G$.~\qed
\end{lem}

Note that if $C:=\partial(X)$ is a separating cut of a \mcg~$G$ then
each of the induced subgraphs $G[X]$~and~$G[\overline{X}]$ is connected.
The following characterization of separating cuts is easily proved;
see \cite[Lemma 2.19]{clm02}.

\begin{lem}
\label{lem:sep-cut-characterization}
A cut~$C$ of a \mcg~$G$ is a separating cut if and only if, for each $e \in E(G)$,
there exists a \pema~$M$ such that $e \in M$ and $|M \cap C|=1$. \qed
\end{lem}

Let $C:=\partial(X)$ denote a separating cut of a \mcg~$G$ such that
the induced (connected) subgraph $G[X]$ is bipartite. Using
Lemma~\ref{lem:sep-cut-characterization}, one may
easily verify that one of the color classes of~$G[X]$ has precisely one more vertex
than the other color class; furthermore, all edges of~$C$ are incident with the bigger color
class; consequently, the $C$-contraction~$G/\overline{X}$
is a bipartite (\mc) graph. This proves the following.

\begin{lem}
\label{lem:sep-cut-contraction-bip}
For a separating cut~$C:=\partial(X)$ of a \mcg~$G$,
the \mbox{$C$-contraction} $G/\overline{X}$
is bipartite if and only if the induced subgraph~$G[X]$ is bipartite. \qed
\end{lem}

\subsection{Tight splicing, tight cuts, bricks and braces}
\label{sec:tight-cuts-bricks-braces}

A cut~$C$ of a \mcg~$G$ is a {\it tight cut} if $|M \cap C|=1$ for each \pema~$M$
of~$G$. It follows from Lemma~\ref{lem:sep-cut-characterization}
that each tight cut is a separating cut.
For instance, if $B$ is a barrier of a \mcg~$G$,
and if~$K$ is a component of~$G-B$,
then~$\partial(V(K))$ is a tight cut of~$G$;
such a tight cut is referred to as a {\it barrier cut associated with the barrier~$B$},
or simply a {\it barrier cut}.

\smallskip
If a \mcg~$G$ is obtained by splicing two smaller {\mcg}s, say~$G_1$~and~$G_2$,
and if the corresponding
splicing cut (i.e., separating cut) is a tight cut (in~$G$) then we also say
that $G$ is obtained by {\it tight splicing $G_1$ and $G_2$}.

\smallskip
In the case of bipartite graphs, one may prove using
Lemmas~\ref{lem:sep-cut-characterization}~and~\ref{lem:sep-cut-contraction-bip}
that each separating cut is indeed a tight cut;
furthermore, one may infer that each tight cut is a barrier cut; see Figure~\ref{fig:bip-tight-cut}.
However, in general, a separating cut need not be a tight cut.
For instance, the separating cuts, depicted by thick lines, in
Figures~\ref{fig:C6bar}~and~\ref{fig:Bicorn}, are not tight cuts.
On the other hand, the separating cut shown in Figure~\ref{fig:K4-splice-K33} is a barrier cut,
and the corresponding splicing is a tight splicing.

\smallskip
For tight cuts, we have the following stronger conclusion
(in comparison to Lemma~\ref{lem:sep-cut-contraction-pm-extension}).

\begin{lem}
\label{lem:tight-cut-contraction-pm-restriction}
Let $C$ denote a tight cut of a \mcg~$G$, and let $G_1$ denote a $C$-contraction of~$G$.
Then each \pema\ of~$G_1$ extends to a \pema\ of~$G$; furthermore, the restriction of each
\pema\ of~$G$ to the set~$E(G_1)$ is a \pema\ of~$G_1$. \qed
\end{lem}

A \mcg~$G$ that is devoid of nontrivial tight cuts is called a {\it brick} if it is nonbipartite,
or a {\it brace} if it is bipartite.
A brick is {\it solid} if it is devoid of nontrivial separating cuts. Thus,
bricks and braces are precisely those {\mcg}s that are free of nontrivial tight cuts;
whereas braces and solid bricks are precisely those {\mcg}s that are free of nontrivial
separating cuts.

\smallskip
The smallest simple bipartite \mcg{s} are the braces $K_2$ and $C_4$,
whereas the smallest nonbipartite ones are the solid brick~$K_4$
and the nonsolid brick~$\overline{C_6}$.
Every \mcg\ of order at least four is $2$-connected. It is easy to see that if a \mcg\
of order six or more has a $2$-vertex-cut then it has a nontrivial tight cut.
This proves the following fact.

\begin{prp}
\label{prp:bricks-braces-3conn}
Every brace of order six or more, and every brick, is $3$-connected. \qed
\end{prp}

Using the fact that, in a bipartite \mcg, every tight cut is a barrier cut
(see Figure~\ref{fig:bip-tight-cut}),
one may obtain the following characterizations of braces.

\begin{prp}
\label{prp:brace-characterization}
For a bipartite \mcg~$H[A,B]$, the following statements are equivalent:
\begin{enumerate}[(i)]
\item $H$ is a brace.
\item $|N_H(S)| \geq |S|+2$ for every nonempty subset~$S$ of~$A$ such that
$|S| < |A|-1$.
\item $H-\{a_1,a_2,b_1,b_2\}$ is matchable for any four distinct vertices $a_1,a_2 \in A$
and $b_1,b_2 \in B$.~\qed
\end{enumerate}
\end{prp}

The first statment of the next corollary follows immediately from
Corollary~\ref{cor:bip-ec} and Proposition~\ref{prp:bricks-braces-3conn},
whereas the second statement follows from
Propositions~\ref{prp:bip-mcg-characterization}~and~\ref{prp:brace-characterization}.

\begin{cor}
\label{cor:brace-removability}
For any brace~$H$, the following statements hold:
\begin{enumerate}[(i)]
\item $\varepsilon(H) \in \{1,2\}$, and $\varepsilon(H) = 2$ if and only if $H$~is~$C_4$ (up to multiple edges) and $H$ has an even $2$-cut.
\item If $|V(H)| \geq 6$ then each edge is removable. \qed
\end{enumerate}
\end{cor}

Lov{\'a}sz~\cite{lova87} proved that if $e$ and $f$ are distinct mutually dependent edges
of a brick~$G$, then $H:=G-e-f$ is a (connected) bipartite matchable graph,
both ends of~$e$ lie in one color class of~$H$ and both ends of $f$ lie in the other color class;
consequently, $\{e,f\}$ is the complement of a cut of~$G$.
(The reader may verify this for each of the doubleton equivalence classes shown in
Figure~\ref{fig:nonbip}.)
Using this fact and Proposition~\ref{prp:bricks-braces-3conn}, one may easily deduce
the following; see\cite[Lemma~2.4]{clm99}.

\begin{cor}
\label{cor:brick-ec}
In any brick~$G$, each \ec\ is either a singleton or a doubleton.
Consequently, $\varepsilon(G) \in \{1,2\}$. \qed
\end{cor}

Unlike the case of braces, the existence (and distribution)
of removable edges is much harder to explain in the case of bricks.
It was shown by Lov{\'a}sz~\cite{lopl86} that every brick,
distinct from~$K_4$ and the triangular prism~$\overline{C_6}$, has a removable edge;
the brick shown in Figure~\ref{fig:Bicorn} has a unique removable edge.

\subsection{Tight cut decomposition, and the invariants $b(G)$ and $c_4(G)$}
\label{sec:tight-cut-separating-cut-decompositions}

Now, let $G$ denote any \mcg. We may apply to~$G$ a recursive procedure, called the
{\it \tcd\ procedure}, to output a list of bricks and braces.
If $G$ is free of nontrivial tight cuts, then this list comprises $G$.
Otherwise, we choose a nontrivial tight cut~$C$ and obtain the two
$C$-contractions (of~$G$), say $G_1$ and $G_2$; these
are smaller {\mcg}s; we now apply the \tcd\ procedure recursively
to each of~$G_1$ and $G_2$, and then combine the resulting output lists into
a single list --- that is the output of an application of the \tcd\ procedure to~$G$.

\smallskip
Lov{\'a}sz~\cite{lova87}
proved the following remarkable property of the \tcd\ procedure,
and used it as one of several ingredients in his characterization of the matching lattice.

\begin{thm}
\label{thm:uniqueness-tight-cut-decomposition}
Any two applications of the \tcd\ procedure to a \mcg\ yield
the same list of bricks and braces (up to multiple edges).
\end{thm}

In light of this, we may now define a couple of useful invariants of a \mcg~$G$.
As usual, we let $b(G)$ denote the number of bricks
yielded by any \tcd\ of~$G$.
(This invariant plays a key role in the theory of {\mcg}s; for instance, it appears
in the formula for the dimension of the \pema\ polytope.)
The following is an immediate consequence of Lemma~\ref{lem:sep-cut-contraction-bip}.

\begin{prp}
\label{prp:bip-iff-zero-bricks}
Let $G$ denote a \mcg.
Given any tight cut~$C$ of~$G$,
the graph~$G$ is bipartite if and only if both $C$-contractions of~$G$ are bipartite.
Consequently, $G$~is bipartite if and only if $b(G)=0$. \qed
\end{prp}

Now, we prove an easy lemma that will be useful to us in Section~\ref{sec:upper-bounds}.
\begin{lem}
\label{lem:tight-cut-both-contractions-nonbipartite-or-one-brace}
Let $G$ be a \mcg\ that is neither a brick nor a brace. Then:
\begin{enumerate}[(i)]
\item either $G$ has a nontrivial tight cut~$C$ such that both \mbox{$C$-contractions}
are nonbipartite,
\item or $G$ has a nontrivial tight cut~$D$ such that one of the \mbox{$D$-contractions} is a brace.
\end{enumerate}
\end{lem}
\begin{proof}
Let $G$ be a \mcg\ that is neither a brick nor a brace;
whence $G$ has a nontrivial tight cut, say~$C:=\partial(X)$.
Let $G_1 := G/\overline{X}$ and $G_2:=G/X$ denote the two \mbox{$C$-contractions} of~$G$.
If $G_1$ and $G_2$ are both nonbipartite then there is nothing to prove.

\smallskip
Now suppose that one of $G_1$ and $G_2$ is bipartite; adjust notation so that $G_1$
is bipartite.
If $G_1$ is a brace then there is nothing to prove. Now suppose that $G_1$ is not a brace;
consequently, $G_1$ has nontrivial tight cut(s).
We may choose a shore~$Y$ of a nontrivial tight cut (in~$G_1$) so that
(i) $\overline{x} \notin Y$ and
(ii) no proper subset of~$Y$ is a shore of a nontrivial tight cut (in~$G_1$).
Thus $D:=\partial(Y)$ is a nontrivial tight cut of~$G_1$, and also of~$G$.
By Lemma~\ref{lem:sep-cut-contraction-bip}, the \mbox{$D$-contraction}
$H:=G/\overline{Y}$ is a bipartite \mcg; by our choice of~$Y$, the graph~$H$
is free of nontrivial tight cuts. Thus $H$ is a brace. This completes the proof.
\end{proof}

For reasons that will be evident later,
we let $c_4(G)$ denote the number of braces isomorphic to~$C_4$ (up to multiple edges)
yielded by any \tcd\ of~$G$.
The following proposition will also be useful to us in Section~\ref{sec:upper-bounds}.

\begin{prp}
\label{prp:even-2-cut-existence}
Let $G$ denote a \mcg\ of order four or more. Then $G$ has an even $2$-cut
if and only if the following holds: there exists an application of the \tcd\ procedure to~$G$
that yields a brace $J$ that is $C_4$ (up to multiple edges) and has an even $2$-cut.
Consequently, if $G$ is free of even $2$-cuts then $c_4(G)=0$.
\end{prp}
\begin{proof}
Let $G$ denote a \mcg\ of order four or more.

\smallskip
First suppose that $G$ has an even $2$-cut, say
$\partial(X):=\{e,f\}$.
Since $G$ is $2$-connected,
$e$~and~$f$ are nonadjacent; furthermore, $e \xleftrightarrow{G} f$.
We let $e:=v\overline{v}$ so that $v \in X$ and $\overline{v} \in \overline{X}$.
Observe that $C:=\partial(X-v)$ and $D:=\partial(\overline{X} - \overline{v})$
are laminar tight cuts of~$G$. Let $J$ denote the graph obtained by first contracting the
shore~$X-v$ (of~$C$) and then contracting the shore~$\overline{X}-\overline{v}$ (of~$D$);
that is, $J:= (G/(X-v))/(\overline{X}-\overline{v})$. Observe that $J$ is indeed the brace~$C_4$
and that $\{e,f\}$ is an even $2$-cut of~$J$. This proves the forward implication.

\smallskip
Now suppose that there exists an application of the tight cut decomposition procedure to~$G$
that yields a brace~$J$ that is $C_4$ (up to multiple edges) and has an even $2$-cut, say~$F$.
By repeatedly invoking Lemma~\ref{lem:preservation-of-even-cuts},
we infer that $F$ is an even $2$-cut of~$G$ as well.
This completes the proof of Proposition~\ref{prp:even-2-cut-existence}.
\end{proof}

\smallskip
One may conveniently define the less well-known {\it \scd\ procedure} by
replacing each occurrence of `tight cut' by `separating cut' in the first paragraph of
this section~(\ref{sec:tight-cut-separating-cut-decompositions}).
It follows from the discussion in Section~\ref{sec:tight-cuts-bricks-braces}
that the output of any application
of the \scd\ procedure (to a \mcg) is a list of braces and
solid bricks. However, unlike the \tcd\ procedure,
two distinct applications of the \scd\ procedure to a matching covered graph
need not yield the same list of braces and solid bricks; in fact, they may even
yield lists of different cardinalities.

\subsection{Our results}

This article is inspired by the following basic question.

\begin{que}
\label{que:main}
If a \mcg~$G$ is obtained by splicing (or by tight splicing)
two smaller {\mcg}s, say~$G_1$~and~$G_2$,
then how is $\mathcal{E}_G$ related to~$\mathcal{E}_{G_1}$ and to~$\mathcal{E}_{G_2}$
(and vice versa)?
\end{que}

In Section~\ref{sec:ec-and-sep-cuts}, we answer the above question with respect to
the splicing operation (or equivalently, with respect to separating cuts).
In Section~\ref{sec:ec-and-tight-cuts}, we answer the above question with respect to
the tight splicing operation
(or equivalently, with respect to tight cuts).
In Section~\ref{sec:applications}, we provide two applications of our findings.

\smallskip
In Section~\ref{sec:upper-bounds}, we establish tight upper bounds on $\varepsilon(G)$;
in particular, we prove that for any \mcg~$G$:
if $G$ is bipartite then $\varepsilon(G) \leq 1 + c_4(G)$;
whereas if $G$ is nonbipartite then $\varepsilon(G) \leq 2 \cdot b(G) + c_4(G)$.

\smallskip
In Section~\ref{sec:arbitrarily-high-kappa-and-varepsilon}, we describe a procedure
to construct {\mcg}s that have arbitrarily high $\kappa(G)$ and $\varepsilon(G)$
simultaneously (where $\kappa(G)$ denotes the vertex-connectivity of~$G$).
This affirmatively answers a recent question of
He, Wei, Ye and Zhai~\cite{hwyz19}.

\section{Equivalence classes and separating cuts}
\label{sec:ec-and-sep-cuts}

Throughout this section, we let $C:=\partial(X)$ denote a nontrivial
separating cut of a \mcg~$G$, and we let $G_1:=G/ \overline{X} \rightarrow \overline{x}$
and $G_2:=G/X \rightarrow x$ denote its $C$-contractions.
Equivalently, $G$ is obtained by splicing the two smaller {\mcg}s $G_1$~and~$G_2$.
The following is an immediate consequence of Lemma~\ref{lem:sep-cut-characterization}.

\begin{lem}
\label{lem:sep-cut-ec-intersection}
If $e \xrightarrow{G} f$ then $| \{e,f\} \cap C | \in \{0,1\}$.
Consequently, each \ec\ of~$G$ meets $C$ in at most one edge. \qed
\end{lem}

By Lemma~\ref{lem:sep-cut-contraction-pm-extension}, for $i \in \{1,2\}$, each \pema\ of~$G_i$
may be extended to a \pema\ of~$G$;
equivalently, each \pema\ of~$G_i$ may be viewed as a restriction of some \pema\
of~$G$ to the set~$E(G_i)$.
Using this, one may infer the following.

\begin{lem}
\label{lem:dependence-inheritance-in-contraction}
For $i \in \{1,2\}$, for any two edges $e,f \in E(G_i)$: if $e \xrightarrow{G} f$ then
$e \xrightarrow{G_i} f$. \qed
\end{lem}

Lemma~\ref{lem:dependence-inheritance-in-contraction}
implies the first part of the following proposition,
whereas Lemma~\ref{lem:sep-cut-ec-intersection} implies the second part.

\begin{prp}
\label{prp:restriction-of-ec-to-contraction}
For each $F \in \mathcal{E}_G$ and for $i \in \{1,2\}$,
the set~$F_i := F \cap E(G_i)$ is a
(not necessarily proper) subset of some member of~$\mathcal{E}_{G_i}$.
Furthermore, $|F| = |F_1| + |F_2| - |F \cap C|$
where $|F \cap C| \in \{0,1\}$.~\qed
\end{prp}

Consequently, the function~$\varepsilon(G)$ satisfies the following subadditivity property
across separating cuts.

\begin{cor}
\label{cor:vareps-subadditivity-across-sep-cuts}
Let $C$ denote a separating cut of a \mcg~$G$, and let $G_1$~and~$G_2$ denote its
$C$-contractions. Then $\varepsilon(G) \leq \varepsilon(G_1) + \varepsilon(G_2)$.~\qed
\end{cor}

\section{Equivalence classes and tight cuts}
\label{sec:ec-and-tight-cuts}

Throughout this section, we let $C:=\partial(X)$ denote a nontrivial
tight cut of a \mcg~$G$, and we let $G_1:=G/ \overline{X} \rightarrow \overline{x}$
and $G_2:=G/X \rightarrow x$ denote its $C$-contractions.
Using Lemma~\ref{lem:tight-cut-contraction-pm-restriction}, we have
the following stronger conclusion
(in comparison to Lemma~\ref{lem:dependence-inheritance-in-contraction}).

\begin{lem}
\label{lem:dependence-preservation-tight-cuts}
For $i \in \{1,2\}$, for any two edges $e,f \in E(G_i)$, $e \xrightarrow{G} f$ if and only if
$e \xrightarrow{G_i} f$.~\qed
\end{lem}

This yields the following two consequences;
the former is a strengthening of Proposition~\ref{prp:restriction-of-ec-to-contraction}
that is applicable in the case of tight cuts; the latter may be viewed as its converse.

\begin{prp}
\label{prp:restriction-of-ec-to-contraction-tight-cut-case}
For each $F \in \mathcal{E}_G$ and for $i \in \{1,2\}$,
if $F_i:=F \cap E(G_i)$ is nonempty
then $F_i$ is a member of~$\mathcal{E}_{G_i}$.
Furthermore, $|F| = |F_1| + |F_2| - |F \cap C|$
where $|F \cap C| \in \{0,1\}$.~\qed
\end{prp}

\begin{prp}
For $i \in \{1,2\}$, let $F_i \in \mathcal{E}_{G_i}$.
Then, for $i \in \{1,2\}$, $F_i$ is a
(not necessarily proper)
subset of some
member of~$\mathcal{E}_G$. Furthermore, if $F_1 \cap F_2$ is nonempty
then (i) $F_1$~and~$F_2$ meet in precisely one edge that lies in the cut~$C$,
(ii) $F:=F_1 \cup F_2$ is a member of~$\mathcal{E}_G$, and (iii) $|F| = |F_1| + |F_2|-1$.~\qed
\end{prp}

The following provides an alternative understanding of the second part of the above
proposition. Let $e$ denote any edge of the tight cut~$C$. For $i \in \{1,2\}$,
let~$F_i$ denote the \ec\ of~$G_i$ that contains $e$;
note that $F_i \cap C = \{e\}$ and that $F_i$ may be a singleton \ec.
Then $F:=F_1 \cup F_2$ is an \ec\ of~$G$ and $F \cap C = \{e\}$.
In other words, the \ec~$F$ (of~$G$) is obtained by ``merging'' an
\ec\ of~$G_1$ with an \ec\ of~$G_2$. In fact, this is how each \ec\ of~$G$,
that meets the cut~$C$, is formed.
This raises the following question.

\begin{que}
\label{que:merging-of-equivalence-classes}
Is it possible for an \ec~$F_1$ of~$G_1$
to merge (i.e., union)
with an \ec~$F_2$ of~$G_2$ to form an \ec~$F$ of~$G$ that does \underline{not}
meet the cut~$C$?
\end{que}

In what follows, we shall answer the above question in the affirmative; in fact,
we will obtain a complete characterization of the circumstances under which this
phenomenon (of merging) occurs. Observe that it suffices to answer the following question.
Under what circumstances do two edges, one in $E(G_1) - C$ and
another in $E(G_2) - C$, become mutually dependent in the graph~$G$?

\smallskip
Henceforth, for $i \in \{1,2\}$, we let $f_i$ denote any edge in~$E(G_i) - C$, and we let~$C_i$
denote the (nonempty) subset of~$C$ that comprises edges that participate in some \pema\
of~$G_i$ with~$f_i$~; in other words,
$C_i := \{e \in C : \{e,f_i\} {\rm ~extends~to~a~perfect~matching~of~} G_i\}$.
Using these definitions, the reader may now verify the following.

\begin{lem}
\footnote{By symmetry: the analogous statement ---
obtained by interchanging all subscripts ``1'' and ``2'' --- holds.}The following are equivalent:
\begin{enumerate}[(i)]
\item $f_1 \xrightarrow{G} f_2$.
\item $e \xrightarrow{G_2} f_2$ for each $e \in C_1$. \qed
\end{enumerate}
\end{lem}

Observe that if each member of~$C_1$ depends on~$f_2$ in the graph~$G_2$
then $C_1 \subseteq C_2$.
Likewise, if each member of~$C_2$ depends on~$f_1$ in the graph~$G_1$
then $C_2 \subseteq C_1$.
This leads us to our next conclusion.

\begin{cor}
\label{cor:mutual-dependence-in-G-characterization}
The following are equivalent:
\begin{enumerate}[(i)]
\item $f_1 \xleftrightarrow{G} f_2$.
\item $C_1 = C_2$ and each of its members is inadmissible in~$G_i - f_i$ for each $i \in \{1,2\}$. \qed
\end{enumerate}
\end{cor}

We may now easily extend this to two {\ec}es $F_1$ and $F_2$ (instead of two edges)
in order to answer Question~\ref{que:merging-of-equivalence-classes}.

\begin{cor}
\label{cor:merging-of-equivalence-classes}
For $i \in \{1,2\}$, let $F_i \in \mathcal{E}_{G_i}$ such that $F_i \cap C = \emptyset$,
and let $C_i$ denote the
set $\{e \in C : F_i \cup \{e\} {\rm ~extends~to~a~perfect~matching~of~} G_i\}$.
Then $F:=F_1 \cup F_2$ is an \ec\ of~$G$ if and only if the following hold:
\begin{enumerate}[(i)]
\item $C_1 = C_2$ and this set has cardinality two or more.
\item For $i \in \{1,2\}$, each member of~$C_i$ is inadmissible in the graph~$G_i - F_i$. \qed
\end{enumerate}
\end{cor}

The above corollary provides the key insight for constructing graphs
with arbitrarily high $\kappa(G)$ and $\varepsilon(G)$ simultaneously --- by means
of the tight splicing operation --- in Section~\ref{sec:arbitrarily-high-kappa-and-varepsilon}.
The following tool will come in handy in Section~\ref{sec:upper-bounds} ---
where we establish tight upper bounds on $\varepsilon(G)$.

\begin{cor}
\label{cor:one-contraction-brace-case}
Assume that $G_1$ is a brace.
Let $F \in \mathcal{E}_G$ such that $F \cap C = \emptyset$ and
\mbox{$F_i := F \cap E(G_i)$} is nonempty for each $i \in \{1,2\}$.
Then $G_1$ is $C_4$ (up to multiple edges), and $|F_1| = 1$.
\end{cor}
\begin{proof}
For each $i \in \{1,2\}$, $F_i \cap C = \emptyset$;
by Proposition~\ref{prp:restriction-of-ec-to-contraction-tight-cut-case},
$F_i \in \mathcal{E}_{G_i}$.
We let $C_1$ denote the set
$\{e \in C : F_1 \cup \{e\} {\rm ~extends~to~a~perfect~matching~of~} G_1\}$.
Since $F:=F_1 \cup F_2$ is a member of~$\mathcal{E}_G$, we may invoke
Corollary~\ref{cor:merging-of-equivalence-classes}. Each member of~$C_1$
is inadmissible in the graph~$G_1 - F_1$;
consequently, each member of~$F_1$ is a non-removable edge of the brace~$G_1$.
It follows from Corollary~\ref{cor:brace-removability} that $|V(G_1)|=4$; whence $G_1$
is $C_4$ (up to multiple edges).
Since $F_1$ is an \ec\ of~$G_1$ that does not meet the cut~$C$, we infer that $F_1$
is a singleton \ec.
Thus $|F_1| = 1$.
\end{proof}

\section{Applications}
\label{sec:applications}

\subsection{Upper bounding $\varepsilon(G)$}
\label{sec:upper-bounds}

In this section, we will establish tight upper bounds on $\varepsilon(G)$ in terms
of the invariants $b(G)$~and~$c_4(G)$.
We begin with the class of bipartite graphs (i.e., precisely those graphs for which
$b(G) = 0$).

\begin{prp}
\label{prp:bip-upper-bound}
For every bipartite \mcg~$G$, the following inequality holds:
$\varepsilon(G) \leq 1 + c_4(G)$.
\end{prp}
\begin{proof}
Let $G$ denote any bipartite \mcg.
We proceed by induction on the order of~$G$.
If $G$ is a brace then the desired inequality holds
due to Corollary~\ref{cor:bip-ec}.

\smallskip
Now suppose that $G$ is not a brace.
By Proposition~\ref{prp:bip-iff-zero-bricks}
and Lemma~\ref{lem:tight-cut-both-contractions-nonbipartite-or-one-brace},
$G$ has a nontrivial tight cut~$C:=\partial(X)$ so that both $C$-contractions are
bipartite {\mcg}s and one of them is a brace.
Adjust notation so that $G_1:=G/\overline{X}$ is a brace, and $G_2:=G/X$
is a bipartite \mcg.
By the induction hypothesis, for each $i \in \{1,2\}$,
the inequality $\varepsilon(G_i) \leq 1 + c_4(G_i)$ holds.
Also, \mbox{$c_4(G) = c_4(G_1) + c_4(G_2)$.}

\smallskip
We let $F \in \mathcal{E}_G$. Our goal is to deduce that $|F| \leq 1 + c_4(G)$.
By Proposition~\ref{prp:restriction-of-ec-to-contraction-tight-cut-case},
for each $i \in \{1,2\}$,
the set~$F_i := F \cap E(G_i)$ is a member of~$\mathcal{E}_{G_i}$.
Thus $|F_i| \leq 1 + c_4(G_i)$.
Furthermore, $|F| = |F_1| + |F_2| - |F \cap C|$ where $|F \cap C| \in \{0,1\}$.
Note that, if $F_1 = F$ or if $F_2 = F$, then the desired inequality holds immediately.

\smallskip
Now suppose that each of $F_1$~and~$F_2$ is a proper subset of~$F$.
Note that, if $|F \cap C| = 1$, then $|F| = |F_1| + |F_2| - 1 \leq 1 + c_4(G_1) + 1 + c_4(G_2) - 1
= 1 + c_4(G)$; the desired inequality holds.
Now suppose that $F \cap C = \emptyset$.
Since $G_1$ is a brace, it follows from Corollary~\ref{cor:one-contraction-brace-case}
that $G_1$ is $C_4$ (up to multiple edges) and $|F_1| = 1$; whence
$|F| = |F_1| + |F_2| \leq 1 + 1 + c_4(G_2) = 1 + c_4(G_1) + c_4(G_2) = 1 + c_4(G)$.
This completes the proof of Proposition~\ref{prp:bip-upper-bound}.
\end{proof}

The upper bound established in Proposition~\ref{prp:bip-upper-bound}
is tight since $c_4(C_{2n}) = n-1$ and $\varepsilon(C_{2n}) = n$,
where $C_{2n}$ is the simple even cycle of order~$2n$; one may easily
construct other such examples.
We now move on to the class of nonbipartite graphs.

\begin{prp}
\label{prp:nonbip-upper-bound}
For every nonbipartite \mcg~$G$, the following inequality holds:
$\varepsilon(G) \leq 2 \cdot b(G) + c_4(G)$.
Furthermore, if $G$ is free of even $2$-cuts then $\varepsilon(G) \leq 2 \cdot b(G)$.
\end{prp}
\begin{proof}
Let $G$ denote any nonbipartite \mcg. We proceed by induction on the order of~$G$.
If $G$ is a brick then the desired inequality holds due to Corollary~\ref{cor:brick-ec}.
Now suppose that $G$ is not a brick; whence $G$ has nontrivial tight cut(s).

\smallskip
First consider the case in which $G$ has a nontrivial tight cut~$C$ so that
both \mbox{$C$-contractions}, say~$G_1$ and $G_2$, are nonbipartite.
By the induction hypothesis, $\varepsilon(G_i) \leq 2 \cdot b(G_i) + c_4(G_i)$
for each $i \in \{1,2\}$. By subadditivity (Corollary~\ref{cor:vareps-subadditivity-across-sep-cuts}):
$\varepsilon(G) \leq \varepsilon(G_1) + \varepsilon(G_2)
\leq 2 \cdot b(G_1) + c_4(G_1) + 2 \cdot b(G_2) + c_4(G_2) = 2 \cdot b(G) + c_4(G)$.

\smallskip
Now consider the case in which $G$ does not have a nontrivial tight cut~$C$
such that both \mbox{$C$-contractions} are nonbipartite.
By Proposition~\ref{prp:bip-iff-zero-bricks}
and Lemma~\ref{lem:tight-cut-both-contractions-nonbipartite-or-one-brace},
$G$ has a nontrivial tight cut~$D:=\partial(X)$ so that
one of the $D$-contractions is a brace, whereas the other is a nonbipartite \mcg.
Adjust notation so that $G_1:=G/\overline{X}$ is a brace, and $G_2:=G/X$
is nonbipartite. Note that $b(G) = b(G_2)$ and that $c_4(G) = c_4(G_1) + c_4(G_2)$.
By the induction hypothesis,
$\varepsilon(G_2) \leq 2 \cdot b(G_2) + c_4 (G_2) = 2 \cdot b(G) + c_4 (G_2)$.
Also, since $G_1$ is bipartite, Proposition~\ref{prp:bip-upper-bound} implies that
$\varepsilon(G_1) \leq 1 + c_4(G_1)$.

\smallskip
We let $F \in \mathcal{E}_G$. Our goal is to deduce that $|F| \leq 2 \cdot b(G) + c_4(G)$.
By Proposition~\ref{prp:restriction-of-ec-to-contraction-tight-cut-case},
for each $i \in \{1,2\}$, the set~$F_i := F \cap E(G_i)$
is a member of~$\mathcal{E}_{G_i}$.
Also, $|F| = |F_1| + |F_2| - |F \cap C|$ where $|F \cap C| \in \{0,1\}$.
Note that if $F = F_1$ then $|F| \leq 1 + c_4(G_1) < 2 \cdot b(G) + c_4(G)$,
and if $F = F_2$ then $|F| \leq 2 \cdot b(G) + c_4(G_2) \leq 2 \cdot b(G) + c_4(G)$.

\smallskip
Now suppose that each of $F_1$~and~$F_2$ is a proper subset of~$F$.
Note that, if $|F \cap C| = 1$, then
$|F| = |F_1| + |F_2| - 1 \leq 1 + c_4(G_1) + 2 \cdot b(G) + c_4(G_2) - 1 = 2 \cdot b(G) + c_4(G)$;
the desired inequality holds.
Now suppose that $F \cap C = \emptyset$.
Since $G_1$ is a brace, it follows from Corollary~\ref{cor:one-contraction-brace-case}
that $G_1$ is $C_4$ (up to multiple edges) and $|F_1|=1$; whence
$|F| = |F_1| + |F_2| \leq 1 + 2 \cdot b(G) + c_4(G_2) = c_4(G_1) + 2 \cdot b(G) + c_4(G_2)
= 2 \cdot b(G) + c_4(G)$.
This completes the proof of the first part of Proposition~\ref{prp:nonbip-upper-bound};
the second part follows by invoking Proposition~\ref{prp:even-2-cut-existence}.
\end{proof}

One may view the above proposition as a generalization
of Corollary~\ref{cor:brick-ec}.
The upper bound established in Proposition~\ref{prp:nonbip-upper-bound} is tight
for the simple reason that there exist (infinitely many) bricks with doubleton equivalence classes.
However, it would be interesting to find tight examples (perhaps infinite families)
with arbitrarily high $b(G)$; the technique used
in Section~\ref{sec:arbitrarily-high-kappa-and-varepsilon} might be
helpful in constructing such graphs.

\subsection{Building graphs with arbitrarily high $\kappa(G)$~and~$\varepsilon(G)$}
\label{sec:arbitrarily-high-kappa-and-varepsilon}

In this section, we prove the following that affirmatively answers a recent question
of He, Wei, Ye and Zhai~\cite{hwyz19}.
The reader may find it helpful to see Figure~\ref{fig:arbitrarily-high-kappa-and-varepsilon}
for a demonstration of the construction provided in the proof.

\begin{prp}
\label{prp:arbitrarily-high-kappa-and-varepsilon}
For any pair of positive integers $p$ and $q$, there exists a \mcg~$G$
such that $\kappa(G) \geq p$ and $\varepsilon(G) \geq q$.
\end{prp}
\begin{proof}
Clearly, we may assume that $p \geq 2$ and $q \geq 2$.
We begin by choosing a pair~$(H[A,B],a)$ where
$H[A,B]$ is a $(p+1)$-connected brace\footnote{For instance, one may choose $H$ to be the complete
bipartite graph $K_{p+1,p+1}$.} and $a \in A$ is a fixed vertex.
Now we construct a simple bipartite $G_0[A_0,B_0]$ as follows.
\begin{enumerate}[(i)]
\item We take $q$ disjoint copies of the pair $(H[A,B],a)$:
$(H_i[A_i,B_i],a_i)$, so that $a_i \in A_i$, for each $i \in \{1,2, \dots, q\}$.
We let $A_0 := \cup_{i=1}^{q} A_i$ and $B_0 := \cup_{i=1}^{q} B_i$.
\item For each $i \in \{1,2, \dots, q-1\}$, we add $p$ edges, each joining $a_i$ with some vertex
in~$B_{i+1}$, and we denote this set of edges by~$C_i$.
\item We add an edge~$f_0$ joining $a_q$ with some vertex in~$B_1$.
\end{enumerate}

We let $C^* := C_1 \cup C_2 \cup \cdots \cup C_{q-1}$.
The reader may easily verify the following.
\begin{sta}
\label{sta:G0-bip-mcg}
The simple bipartite graph~$G_0[A_0,B_0]$
is \mc\footnote{Invoke Proposition~\ref{prp:bip-mcg-characterization}},
each member of~$C^*$ is inadmissible in $G_0 - f_0$,
and $\{f_0\} \in \mathcal{E}_{G_0}$. \qed
\end{sta}

For $i \in \{1,2, \dots, q-1\}$, we construct a simple nonbipartite graph~$J_i$ as follows.
\begin{enumerate}[(i)]
\item We take a new copy of the pair~$(H[A,B],a)$, say~$(L_i[U_i,V_i],u_i)$, so that $u_i \in U_i$.
\item We add $p$ edges, each joining $u_i$ with some vertex in~$U_i - u_i$,
and we denote this set of edges by~$C'_i$.
\item We add an edge~$f_i$ that has both ends in~$V_i$.
\end{enumerate}

The reader may easily verify the following.
\begin{sta}
\label{sta:J-bricks}
For $i \in \{1,2,\dots,q-1\}$,
the simple nonbipartite graph~$J_i$
is a brick\footnote{Use Proposition~\ref{prp:brace-characterization} to show that $J_i$
is \mc. Adding edges cannot create ``new'' tight~cuts.},
each member of~$C'_i$ is inadmissible in~$J_i - f_i$,
and $\{f_i\} \in \mathcal{E}_{J_i}$. \qed
\end{sta}

\tikzstyle{every node}=[circle, draw, fill=white,inner sep=0pt, minimum width=4pt]
\begin{figure}[!htb]
\centering
\subfigure[$G_0$]{
\begin{tikzpicture}[scale=0.75]

\draw[ultra thick] (0,2) to [out=60,in=180] (2,3) to [out=0,in=180] (13,3) to [out=0,in=60] (12.75,0);
\draw (14,2.7)node[nodelabel]{$f_0$};

\draw[ultra thick] (3.5,1.1) to [out=0,in=90] (4,0.5);
\draw (4.1,0.2)node[nodelabel]{$C_1$};

\draw[ultra thick] (8.75,1.1) to [out=0,in=90] (9.25,0.5);
\draw (9.35,0.2)node[nodelabel]{$C_2$};


\draw (2.25,0) -- (0,2);
\draw (2.25,0) -- (0.75,2);
\draw (2.25,0) -- (1.5,2);
\draw (2.25,0) -- (2.25,2);
\draw (2.25,0) -- (5.25,2);
\draw (2.25,0) -- (6,2);
\draw (2.25,0) -- (6.75,2);

\draw (7.5,0) -- (5.25,2);
\draw (7.5,0) -- (6,2);
\draw (7.5,0) -- (6.75,2);
\draw (7.5,0) -- (7.5,2);
\draw (7.5,0) -- (10.5,2);
\draw (7.5,0) -- (11.25,2);
\draw (7.5,0) -- (12,2);

\draw (0,0) -- (0,2);
\draw (0,0) -- (0.75,2);
\draw (0,0) -- (1.5,2);
\draw (0,0) -- (2.25,2);
\draw (0.75,0) -- (0,2);
\draw (0.75,0) -- (0.75,2);
\draw (0.75,0) -- (1.5,2);
\draw (0.75,0) -- (2.25,2);
\draw (1.5,0) -- (0,2);
\draw (1.5,0) -- (0.75,2);
\draw (1.5,0) -- (1.5,2);
\draw (1.5,0) -- (2.25,2);

\draw (5.25,0) -- (5.25,2);
\draw (5.25,0) -- (6,2);
\draw (5.25,0) -- (6.75,2);
\draw (5.25,0) -- (7.5,2);
\draw (6,0) -- (5.25,2);
\draw (6,0) -- (6,2);
\draw (6,0) -- (6.75,2);
\draw (6,0) -- (7.5,2);
\draw (6.75,0) -- (5.25,2);
\draw (6.75,0) -- (6,2);
\draw (6.75,0) -- (6.75,2);
\draw (6.75,0) -- (7.5,2);

\draw (10.5,0) -- (10.5,2);
\draw (10.5,0) -- (11.25,2);
\draw (10.5,0) -- (12,2);
\draw (10.5,0) -- (12.75,2);
\draw (11.25,0) -- (10.5,2);
\draw (11.25,0) -- (11.25,2);
\draw (11.25,0) -- (12,2);
\draw (11.25,0) -- (12.75,2);
\draw (12,0) -- (10.5,2);
\draw (12,0) -- (11.25,2);
\draw (12,0) -- (12,2);
\draw (12,0) -- (12.75,2);
\draw (12.75,0) -- (10.5,2);
\draw (12.75,0) -- (11.25,2);
\draw (12.75,0) -- (12,2);
\draw (12.75,0) -- (12.75,2);

\draw (0,0)node{};
\draw (0,2)node{};
\draw (0.75,0)node{};
\draw (0.75,2)node{};
\draw (1.5,0)node{};
\draw (1.5,2)node{};
\draw (2.25,0)node{};
\draw (2.25,2)node{};

\draw (5.25,0)node{};
\draw (5.25,2)node{};
\draw (6,0)node{};
\draw (6,2)node{};
\draw (6.75,0)node{};
\draw (6.75,2)node{};
\draw (7.5,0)node{};
\draw (7.5,2)node{};

\draw (10.5,0)node{};
\draw (10.5,2)node{};
\draw (11.25,0)node{};
\draw (11.25,2)node{};
\draw (12,0)node{};
\draw (12,2)node{};
\draw (12.75,0)node{};
\draw (12.75,2)node{};


\draw (2.25,0)node[nodelabel,below]{$a_1$};
\draw (7.5,0)node[nodelabel,below]{$a_2$};

\draw[ultra thin] (-0.2,1.8) -- (-0.2,2.2) -- (12.95,2.2) -- (12.95,1.8) -- (-0.2,1.8);
\draw (-0.6,2)node[nodelabel]{$B_0$};

\end{tikzpicture}
\label{fig:G0}
}
\subfigure[$J_1$]
{
\begin{tikzpicture}[scale=0.75]

\draw[ultra thick] (1.4,1.85) -- (2.2,1.85);
\draw (2.6,1.85)node[nodelabel]{$C'_1$};

\draw (1.5,2.25) -- (0,0);
\draw (1.5,2.25) -- (0,0.75);
\draw (1.5,2.25) -- (0,1.5);
\draw (1.5,2.25) -- (0,2.25);
\draw (1.5,2.25) to [out=300,in=60] (1.5,1.5);
\draw (1.5,2.25) to [out=315,in=45] (1.5,0.75);
\draw (1.5,2.25) to [out=330,in=30] (1.5,0);

\draw (1.5,0) -- (0,0);
\draw (1.5,0) -- (0,0.75);
\draw (1.5,0) -- (0,1.5);
\draw (1.5,0) -- (0,2.25);
\draw (1.5,0.75) -- (0,0);
\draw (1.5,0.75) -- (0,0.75);
\draw (1.5,0.75) -- (0,1.5);
\draw (1.5,0.75) -- (0,2.25);
\draw (1.5,1.5) -- (0,0);
\draw (1.5,1.5) -- (0,0.75);
\draw (1.5,1.5) -- (0,1.5);
\draw (1.5,1.5) -- (0,2.25);

\draw[ultra thick] (0,0) to [out=135,in=225] (0,0.75);
\draw (-0.5,0.375)node[nodelabel]{$f_1$};

\draw (0,0)node{};
\draw (0,0.75)node{};
\draw (0,1.5)node{};
\draw (0,2.25)node{};
\draw (1.5,0)node{};
\draw (1.5,0.75)node{};
\draw (1.5,1.5)node{};
\draw (1.5,2.25)node{};

\draw (1.5,2.25)node{}node[nodelabel,above]{$u_1$};


\end{tikzpicture}
\label{fig:J1}
}
\subfigure[$G_1$]{
\begin{tikzpicture}[scale=0.75]

\draw[ultra thick] (0,2) to [out=60,in=180] (2,3) to [out=0,in=180] (13,3) to [out=0,in=60] (12.75,0);
\draw (14,2.7)node[nodelabel]{$f_0$};
\draw[ultra thick] (2.75,-3) to [out=135,in=225] (2.75,-2.25);
\draw (2.35,-2.625)node[nodelabel]{$f_1$};

\draw[ultra thick] (1.5,-1) to [out=45,in=135] (5.25,-1);
\draw (5.6,-1.2)node[nodelabel]{$D_1$};

\draw[ultra thick] (8.75,1.1) to [out=0,in=90] (9.25,0.5);
\draw (9.35,0.2)node[nodelabel]{$C_2$};

\draw (0,2) to [out=315,in=90] (1.9,-0.5) to [out=270,in=210] (2.75,-3);
\draw (0.75,2) to [out=330,in=90] (2.25,-0.5) to [out=270,in=150] (2.75,-2.25);
\draw (1.5,2) to [out=345,in=90] (2.5,-0.5) to [out=270,in=120] (2.75,-1.5);
\draw (2.25,2) to [out=360,in=90] (2.75,-0.5) to [out=270,in=90] (2.75,-0.75);
\draw (6.75,2) to [out=210,in=90] (4.75,-0.5) to [out=270,in=0] (4.25,-3);
\draw (6,2) to [out=195,in=90] (4.5,-0.5) to [out=270,in=45] (4.25,-2.25);
\draw (5.25,2) to [out=180,in=90] (4.25,-0.5) to [out=270,in=90] (4.25,-1.5);

\draw (4.25,-3) -- (2.75,-3);
\draw (4.25,-3) -- (2.75,-2.25);
\draw (4.25,-3) -- (2.75,-1.5);
\draw (4.25,-3) -- (2.75,-0.75);
\draw (4.25,-2.25) -- (2.75,-3);
\draw (4.25,-2.25) -- (2.75,-2.25);
\draw (4.25,-2.25) -- (2.75,-1.5);
\draw (4.25,-2.25) -- (2.75,-0.75);
\draw (4.25,-1.5) -- (2.75,-3);
\draw (4.25,-1.5) -- (2.75,-2.25);
\draw (4.25,-1.5) -- (2.75,-1.5);
\draw (4.25,-1.5) -- (2.75,-0.75);

\draw (2.75,-3)node{};
\draw (4.25,-3)node{};
\draw (2.75,-2.25)node{};
\draw (4.25,-2.25)node{};
\draw (2.75,-1.5)node{};
\draw (4.25,-1.5)node{};
\draw (2.75,-0.75)node{};

\draw (7.5,0) -- (5.25,2);
\draw (7.5,0) -- (6,2);
\draw (7.5,0) -- (6.75,2);
\draw (7.5,0) -- (7.5,2);
\draw (7.5,0) -- (10.5,2);
\draw (7.5,0) -- (11.25,2);
\draw (7.5,0) -- (12,2);

\draw (0,0) -- (0,2);
\draw (0,0) -- (0.75,2);
\draw (0,0) -- (1.5,2);
\draw (0,0) -- (2.25,2);
\draw (0.75,0) -- (0,2);
\draw (0.75,0) -- (0.75,2);
\draw (0.75,0) -- (1.5,2);
\draw (0.75,0) -- (2.25,2);
\draw (1.5,0) -- (0,2);
\draw (1.5,0) -- (0.75,2);
\draw (1.5,0) -- (1.5,2);
\draw (1.5,0) -- (2.25,2);

\draw (5.25,0) -- (5.25,2);
\draw (5.25,0) -- (6,2);
\draw (5.25,0) -- (6.75,2);
\draw (5.25,0) -- (7.5,2);
\draw (6,0) -- (5.25,2);
\draw (6,0) -- (6,2);
\draw (6,0) -- (6.75,2);
\draw (6,0) -- (7.5,2);
\draw (6.75,0) -- (5.25,2);
\draw (6.75,0) -- (6,2);
\draw (6.75,0) -- (6.75,2);
\draw (6.75,0) -- (7.5,2);

\draw (10.5,0) -- (10.5,2);
\draw (10.5,0) -- (11.25,2);
\draw (10.5,0) -- (12,2);
\draw (10.5,0) -- (12.75,2);
\draw (11.25,0) -- (10.5,2);
\draw (11.25,0) -- (11.25,2);
\draw (11.25,0) -- (12,2);
\draw (11.25,0) -- (12.75,2);
\draw (12,0) -- (10.5,2);
\draw (12,0) -- (11.25,2);
\draw (12,0) -- (12,2);
\draw (12,0) -- (12.75,2);
\draw (12.75,0) -- (10.5,2);
\draw (12.75,0) -- (11.25,2);
\draw (12.75,0) -- (12,2);
\draw (12.75,0) -- (12.75,2);

\draw (0,0)node{};
\draw (0,2)node{};
\draw (0.75,0)node{};
\draw (0.75,2)node{};
\draw (1.5,0)node{};
\draw (1.5,2)node{};
\draw (2.25,2)node{};

\draw (5.25,0)node{};
\draw (5.25,2)node{};
\draw (6,0)node{};
\draw (6,2)node{};
\draw (6.75,0)node{};
\draw (6.75,2)node{};
\draw (7.5,0)node{};
\draw (7.5,2)node{};

\draw (10.5,0)node{};
\draw (10.5,2)node{};
\draw (11.25,0)node{};
\draw (11.25,2)node{};
\draw (12,0)node{};
\draw (12,2)node{};
\draw (12.75,0)node{};
\draw (12.75,2)node{};


\draw (7.5,0)node[nodelabel,below]{$a_2$};

\draw[ultra thin] (-0.2,1.8) -- (-0.2,2.2) -- (12.95,2.2) -- (12.95,1.8) -- (-0.2,1.8);
\draw (-0.6,2)node[nodelabel]{$B_0$};

\end{tikzpicture}
\label{fig:G1}
}
\subfigure[$J_2$]
{
\begin{tikzpicture}[scale=0.75]

\draw[ultra thick] (1.4,1.85) -- (2.2,1.85);
\draw (2.6,1.85)node[nodelabel]{$C'_2$};

\draw (1.5,2.25) -- (0,0);
\draw (1.5,2.25) -- (0,0.75);
\draw (1.5,2.25) -- (0,1.5);
\draw (1.5,2.25) -- (0,2.25);
\draw (1.5,2.25) to [out=300,in=60] (1.5,1.5);
\draw (1.5,2.25) to [out=315,in=45] (1.5,0.75);
\draw (1.5,2.25) to [out=330,in=30] (1.5,0);

\draw (1.5,0) -- (0,0);
\draw (1.5,0) -- (0,0.75);
\draw (1.5,0) -- (0,1.5);
\draw (1.5,0) -- (0,2.25);
\draw (1.5,0.75) -- (0,0);
\draw (1.5,0.75) -- (0,0.75);
\draw (1.5,0.75) -- (0,1.5);
\draw (1.5,0.75) -- (0,2.25);
\draw (1.5,1.5) -- (0,0);
\draw (1.5,1.5) -- (0,0.75);
\draw (1.5,1.5) -- (0,1.5);
\draw (1.5,1.5) -- (0,2.25);

\draw[ultra thick] (0,0) to [out=135,in=225] (0,0.75);
\draw (-0.5,0.375)node[nodelabel]{$f_2$};


\draw (0,0)node{};
\draw (0,0.75)node{};
\draw (0,1.5)node{};
\draw (0,2.25)node{};
\draw (1.5,0)node{};
\draw (1.5,0.75)node{};
\draw (1.5,1.5)node{};
\draw (1.5,2.25)node{};

\draw (1.5,2.25)node{}node[nodelabel,above]{$u_2$};


\end{tikzpicture}
\label{fig:J2}
}
\subfigure[$G_2$]{
\begin{tikzpicture}[scale=0.75]

\draw[ultra thick] (0,2) to [out=60,in=180] (2,3) to [out=0,in=180] (13,3) to [out=0,in=60] (12.75,0);
\draw (14,2.7)node[nodelabel]{$f_0$};
\draw[ultra thick] (2.75,-3) to [out=135,in=225] (2.75,-2.25);
\draw (2.35,-2.625)node[nodelabel]{$f_1$};
\draw[ultra thick] (8,-3) to [out=135,in=225] (8,-2.25);
\draw (7.60,-2.625)node[nodelabel]{$f_2$};

\draw[ultra thick] (1.5,-1) to [out=45,in=135] (5.25,-1);
\draw (5.6,-1.2)node[nodelabel]{$D_1$};

\draw[ultra thick] (6.75,-1) to [out=45,in=135] (10.5,-1);
\draw (10.85,-1.2)node[nodelabel]{$D_2$};

\draw (5.25,2) to [out=315,in=90] (7.15,-0.5) to [out=270,in=210] (8,-3);
\draw (6,2) to [out=330,in=90] (7.5,-0.5) to [out=270,in=150] (8,-2.25);
\draw (6.75,2) to [out=345,in=90] (7.75,-0.5) to [out=270,in=120] (8,-1.5);
\draw (7.5,2) to [out=360,in=90] (8,-0.5) to [out=270,in=90] (8,-0.75);
\draw (12,2) to [out=210,in=90] (10,-0.5) to [out=270,in=0] (9.5,-3);
\draw (11.25,2) to [out=195,in=90] (9.75,-0.5) to [out=270,in=45] (9.5,-2.25);
\draw (10.5,2) to [out=180,in=90] (9.5,-0.5) to [out=270,in=90] (9.5,-1.5);

\draw (0,2) to [out=315,in=90] (1.9,-0.5) to [out=270,in=210] (2.75,-3);
\draw (0.75,2) to [out=330,in=90] (2.25,-0.5) to [out=270,in=150] (2.75,-2.25);
\draw (1.5,2) to [out=345,in=90] (2.5,-0.5) to [out=270,in=120] (2.75,-1.5);
\draw (2.25,2) to [out=360,in=90] (2.75,-0.5) to [out=270,in=90] (2.75,-0.75);
\draw (6.75,2) to [out=210,in=90] (4.75,-0.5) to [out=270,in=0] (4.25,-3);
\draw (6,2) to [out=195,in=90] (4.5,-0.5) to [out=270,in=45] (4.25,-2.25);
\draw (5.25,2) to [out=180,in=90] (4.25,-0.5) to [out=270,in=90] (4.25,-1.5);

\draw (9.5,-3) -- (8,-3);
\draw (9.5,-3) -- (8,-2.25);
\draw (9.5,-3) -- (8,-1.5);
\draw (9.5,-3) -- (8,-0.75);
\draw (9.5,-2.25) -- (8,-3);
\draw (9.5,-2.25) -- (8,-2.25);
\draw (9.5,-2.25) -- (8,-1.5);
\draw (9.5,-2.25) -- (8,-0.75);
\draw (9.5,-1.5) -- (8,-3);
\draw (9.5,-1.5) -- (8,-2.25);
\draw (9.5,-1.5) -- (8,-1.5);
\draw (9.5,-1.5) -- (8,-0.75);

\draw (8,-3)node{};
\draw (9.5,-3)node{};
\draw (8,-2.25)node{};
\draw (9.5,-2.25)node{};
\draw (8,-1.5)node{};
\draw (9.5,-1.5)node{};
\draw (8,-0.75)node{};

\draw (4.25,-3) -- (2.75,-3);
\draw (4.25,-3) -- (2.75,-2.25);
\draw (4.25,-3) -- (2.75,-1.5);
\draw (4.25,-3) -- (2.75,-0.75);
\draw (4.25,-2.25) -- (2.75,-3);
\draw (4.25,-2.25) -- (2.75,-2.25);
\draw (4.25,-2.25) -- (2.75,-1.5);
\draw (4.25,-2.25) -- (2.75,-0.75);
\draw (4.25,-1.5) -- (2.75,-3);
\draw (4.25,-1.5) -- (2.75,-2.25);
\draw (4.25,-1.5) -- (2.75,-1.5);
\draw (4.25,-1.5) -- (2.75,-0.75);

\draw (2.75,-3)node{};
\draw (4.25,-3)node{};
\draw (2.75,-2.25)node{};
\draw (4.25,-2.25)node{};
\draw (2.75,-1.5)node{};
\draw (4.25,-1.5)node{};
\draw (2.75,-0.75)node{};

\draw (0,0) -- (0,2);
\draw (0,0) -- (0.75,2);
\draw (0,0) -- (1.5,2);
\draw (0,0) -- (2.25,2);
\draw (0.75,0) -- (0,2);
\draw (0.75,0) -- (0.75,2);
\draw (0.75,0) -- (1.5,2);
\draw (0.75,0) -- (2.25,2);
\draw (1.5,0) -- (0,2);
\draw (1.5,0) -- (0.75,2);
\draw (1.5,0) -- (1.5,2);
\draw (1.5,0) -- (2.25,2);

\draw (5.25,0) -- (5.25,2);
\draw (5.25,0) -- (6,2);
\draw (5.25,0) -- (6.75,2);
\draw (5.25,0) -- (7.5,2);
\draw (6,0) -- (5.25,2);
\draw (6,0) -- (6,2);
\draw (6,0) -- (6.75,2);
\draw (6,0) -- (7.5,2);
\draw (6.75,0) -- (5.25,2);
\draw (6.75,0) -- (6,2);
\draw (6.75,0) -- (6.75,2);
\draw (6.75,0) -- (7.5,2);

\draw (10.5,0) -- (10.5,2);
\draw (10.5,0) -- (11.25,2);
\draw (10.5,0) -- (12,2);
\draw (10.5,0) -- (12.75,2);
\draw (11.25,0) -- (10.5,2);
\draw (11.25,0) -- (11.25,2);
\draw (11.25,0) -- (12,2);
\draw (11.25,0) -- (12.75,2);
\draw (12,0) -- (10.5,2);
\draw (12,0) -- (11.25,2);
\draw (12,0) -- (12,2);
\draw (12,0) -- (12.75,2);
\draw (12.75,0) -- (10.5,2);
\draw (12.75,0) -- (11.25,2);
\draw (12.75,0) -- (12,2);
\draw (12.75,0) -- (12.75,2);

\draw (0,0)node{};
\draw (0,2)node{};
\draw (0.75,0)node{};
\draw (0.75,2)node{};
\draw (1.5,0)node{};
\draw (1.5,2)node{};
\draw (2.25,2)node{};

\draw (5.25,0)node{};
\draw (5.25,2)node{};
\draw (6,0)node{};
\draw (6,2)node{};
\draw (6.75,0)node{};
\draw (6.75,2)node{};
\draw (7.5,2)node{};

\draw (10.5,0)node{};
\draw (10.5,2)node{};
\draw (11.25,0)node{};
\draw (11.25,2)node{};
\draw (12,0)node{};
\draw (12,2)node{};
\draw (12.75,0)node{};
\draw (12.75,2)node{};

\draw[ultra thin] (-0.2,1.8) -- (-0.2,2.2) -- (12.95,2.2) -- (12.95,1.8) -- (-0.2,1.8);
\draw (-0.6,2)node[nodelabel]{$B_0$};

\end{tikzpicture}
\label{fig:G2}
}
\caption{Illustration for the proof of Proposition~\ref{prp:arbitrarily-high-kappa-and-varepsilon} for
$p=3$ and $q=3$ (and $H:=K_{4,4}$)}
\label{fig:arbitrarily-high-kappa-and-varepsilon}
\end{figure}
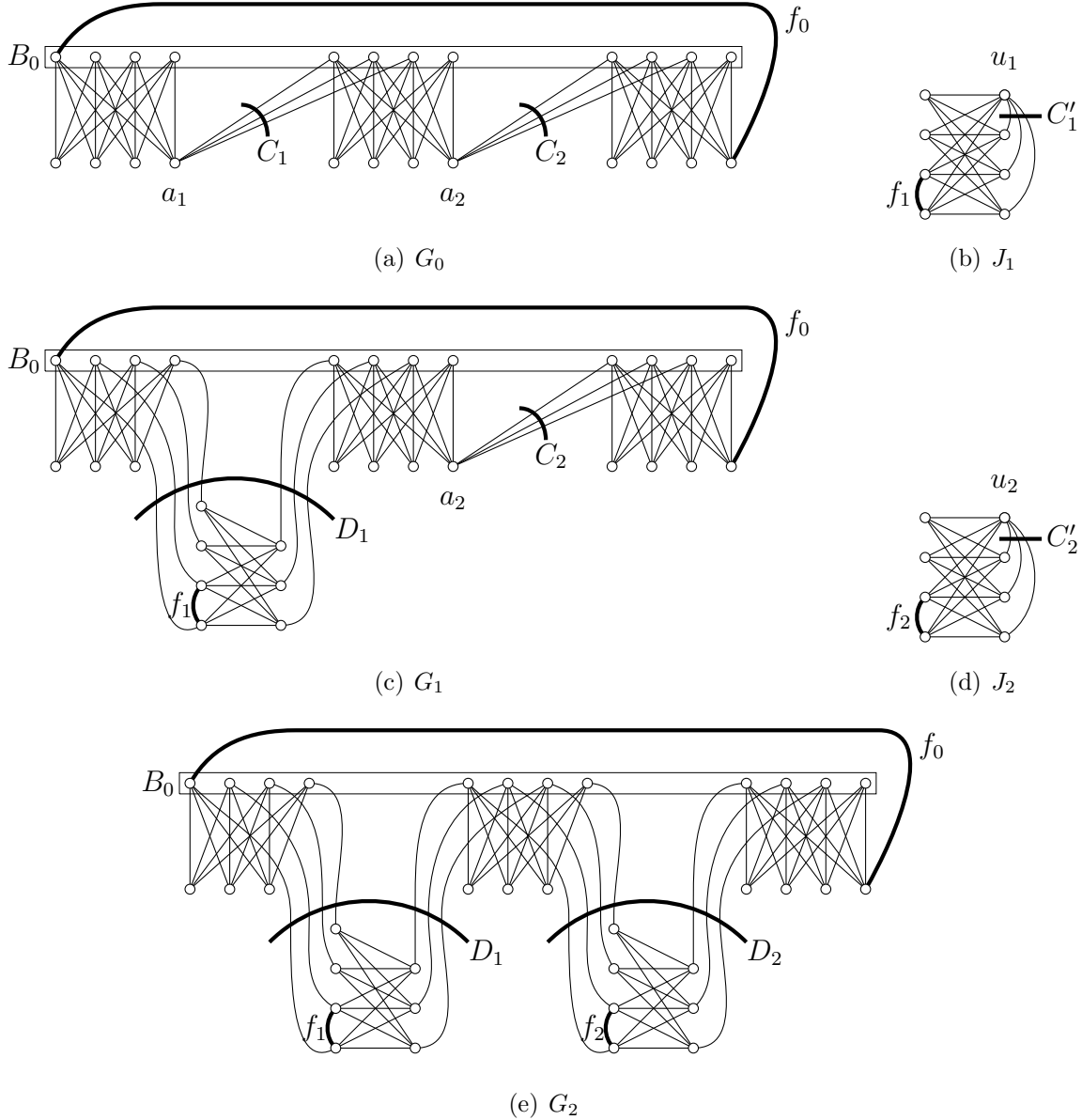

Finally, for $i \in \{1,2, \dots, q-1\}$, we let $G_i := (G_{i-1} \odot J_i)_{a_i,u_i,\pi_i}$
where $\pi_i$ is any bijection between $\partial_{G_{i-1}}(a_i)$ and $\partial_{J_i}(u_i)$
that maps each member of~$C_i$ to some member of~$C'_i$, and we let
$D_i$~denote the corresponding splicing cut of~$G_i$.

\smallskip
By Lemma~\ref{lem:splicing-simple}~and~\ref{lem:splicing-mcg},
all of the graphs~$G_1, G_2, \dots, G_{q-1}$
are simple and \mc; note that $B_0$ is a barrier in each of them. This proves the following.

\begin{sta}
\label{sta:intermediate-graphs}
For $i \in \{1,2,\dots,q\}$, the simple graph~$G_i$ is \mc\ and
$D_i$ is a barrier cut associated with the barrier~$B_0$. \qed
\end{sta}

For each $i \in \{0,1, \dots, q-1\}$, we let $F_i := \{f_0, f_1, \dots, f_i\}$.

\begin{sta}
\label{sta:equivalence-class-growth}
For $i \in \{0,1,2,\dots,q-1\}$, each member of~$C^*$
is inadmissible in~$G_i - F_i$, and $F_i \in \mathcal{E}_{G_i}$.
\end{sta}
\begin{proof}
We proceed by induction on~$i$.
By \ref{sta:G0-bip-mcg}, the desired conclusions hold when $i=0$.

\smallskip
Now suppose that $1 \leq i \leq q-1$, and that the desired conclusions hold for $i-1$.
By the induction hypothesis, each member of~$C^*$
is inadmissible in $G_{i-1} - F_{i-1}$.
That is, if $e \in C^*$
and if $f \in F_{i-1}$, then $e \xrightarrow{G_{i-1}} f$;
consequently, since~$G_i$ is obtained by tight splicing $G_{i-1}$~and~$J_i$,
Lemma~\ref{lem:dependence-preservation-tight-cuts} implies
that $e \xrightarrow{G_i} f$.
In other words, each member of~$C^*$ is inadmissible in~$G_i - F_{i-1}$;
whence each member of~$C^*$ is inadmissible in~$G_i - F_i$.

\smallskip
By the induction hypothesis, $F_{i-1} \in \mathcal{E}_{G_{i-1}}$
and each member of~$C_i$ is inadmissible in \mbox{$G_{i-1} - F_{i-1}$}.
On the other hand, by \ref{sta:J-bricks}, $\{f_i\} \in \mathcal{E}_{J_i}$
and each member of~$C'_i$ is inadmissible in~$J_i - f_i$.
Recall that $G_i := (G_{i-1} \odot J_i)_{a_i,u_i,\pi_i}$ where $\pi_i$ is a bijection
between $\partial_{G_{i-1}}(a_i)$ and $\partial_{J_i}(u_i)$ that maps each member
of~$C_i$ to some member of~$C'_i$;
also, by \ref{sta:intermediate-graphs}, the corresponding splicing cut is a tight cut;
thus we invoke Corollary~\ref{cor:merging-of-equivalence-classes}
to infer that $F_{i-1} \cup \{f_i\}$ is an \ec\ of~$G_i$. In other words,
$F_i \in \mathcal{E}_{G_i}$.
\end{proof}

We let $G:= G_{q-1}$.
By \ref{sta:equivalence-class-growth}, $F_{q-1} \in \mathcal{E}_G$.
Thus $\varepsilon(G) \geq q$. It remains to show that $\kappa(G) \geq p$.
We consider the following ordered partition of~$V(G)$.
\[
((V(H_1) - a_1), (V(L_1) - u_1), (V(H_2) - a_2), (V(L_2) - u_2),\dots,(V(L_{q-1}) - u_{q-1}),V(H_q)).
\]

Since the brace~$H$ is $(p+1)$-connected,
the subgraph (of~$G$) induced by each part (of the above partition) is $p$-connected.
By our construction of~$G$, and by Lemma~\ref{lem:splicing-simple},
there is a matching of cardinality $p$ (or more)
joining any two consecutive parts of the above partition.
Using Menger's Theorem, we infer that $G$ is $p$-connected. Thus $\kappa(G) \geq p$.

\smallskip
This completes the proof of Proposition~\ref{prp:arbitrarily-high-kappa-and-varepsilon}.
\end{proof}

Thus we have shown that there exist highly-connected graphs with arbitrarily large {\ec}es;
this is in stark contrast to {\rc}es as was demonstrated by Lov{\'a}sz and Plummer
(see Theorem~\ref{thm:rc-singleton-or-doubleton}).

\bibliographystyle{plain}
\bibliography{clm}

\begin{thebibliography}{1}

\bibitem{bomu08}
J.~A. Bondy and U.~S.~R. Murty.
\newblock {\em Graph Theory}.
\newblock Springer, 2008.

\bibitem{clm99}
M.~H. Carvalho, C.~L. Lucchesi, and U.~S.~R. Murty.
\newblock Ear decompositions of matching covered graphs.
\newblock {\em Combinatorica}, 19:151--174, 1999.

\bibitem{clm02}
M.~H. Carvalho, C.~L. Lucchesi, and U.~S.~R. Murty.
\newblock On a conjecture of {L}ov{\'a}sz concerning bricks. {I}. {T}he
  characteristic of a matching covered graph.
\newblock {\em J.~Combin.~Theory Ser.~B}, 85:94--136, 2002.

\bibitem{hwyz19}
Jinghua He, Erling Wei, Dong Ye, and Shaohui Zhai.
\newblock On perfect matchings in matching covered graphs.
\newblock {\em J.~Graph Theory}, 90:535--546, 2019.

\bibitem{lova87}
L.~Lov\'asz.
\newblock Matching structure and the matching lattice.
\newblock {\em J.~Combin.~Theory Ser.~B}, 43:187--222, 1987.

\bibitem{lopl86}
L.~Lov\'asz and M.~D. Plummer.
\newblock {\em Matching Theory}.
\newblock Number~29 in Annals of Discrete Mathematics. Elsevier Science, 1986.

\bibitem{lckm18}
C.~L. Lucchesi, M.~H. Carvalho, N.~Kothari, and U.~S.~R. Murty.
\newblock On two unsolved problems concerning matching covered graphs.
\newblock {\em \textsc{siam} J. Discrete Math.}, 32(2):1478--1504, 2018.

\end{thebibliography}
\end{document}